\documentclass[reqno,]{amsart} 

\usepackage{color}
\usepackage{xcolor}

\usepackage{svg} 

\usepackage{ifpdf}
\ifpdf 
\usepackage{graphicx}   
\usepackage[pdftex,     
plainpages=false,   
breaklinks=true,    
colorlinks=true,
linkcolor=blue,
citecolor=green,
pdftitle={Existence and uniqueness for the transport of currents by Lipschitz vector fields},
pdfauthor={Paolo Bonicatto, Giacomo Del Nin, Filip Rindler}
]{hyperref} 
\else 
\usepackage{graphicx}       
\usepackage{hyperref}       
\fi 

\usepackage{amsfonts,amsmath}	
\usepackage{amssymb}
\usepackage{verbatim}
\usepackage{amsopn}
\usepackage[english]{babel}
\usepackage{amsthm}
\usepackage{enumerate}
\usepackage{mathrsfs}	
\usepackage{mathtools}
\usepackage{harpoon}
\usepackage{esint}
\usepackage{todonotes}
\usepackage{enumerate}

\usepackage{stmaryrd}
\usepackage{bm}
\usepackage{bbm}
\usepackage{caption}
\usepackage{subcaption}
\usepackage{nicefrac}
\usepackage{enumitem}

\usepackage[a4paper,left=35mm,right=35mm,top=34mm,bottom=40mm]{geometry} 

\date{\today}

\newcommand{\Crm}{\mathrm{C}}

\newcommand{\Lrm}{\mathrm{L}}
\newcommand{\Mrm}{\mathrm{M}}
\newcommand{\Nrm}{\mathrm{N}}

\newcommand{\Lcal}{\mathcal{L}}
\newcommand{\Mcal}{\mathcal{M}}

\newcommand{\Mbf}{\mathbf{M}}

\newcommand{\absb}[1]{\bigl|#1\bigr|}

\newcommand{\altnorm}[1]{{\left\vert\kern-0.25ex\left\vert\kern-0.25ex\left\vert #1 \right\vert\kern-0.25ex\right\vert\kern-0.25ex\right\vert}}
\newcommand{\eps}{\varepsilon}

\newcommand{\dpr}[1]{\langle #1 \rangle}

\newcommand{\dprb}[1]{\bigl\langle #1 \bigr\rangle}

\newcommand{\tbf}{\mathbf{t}}
\newcommand{\pbf}{\mathbf{p}}

\theoremstyle{definition} \newtheorem{definition}{Definition}[section]
\theoremstyle{definition} \newtheorem{remark}[definition]{Remark}
\theoremstyle{plain} \newtheorem{lemma}[definition]{Lemma}
\theoremstyle{plain} \newtheorem{proposition}[definition]{Proposition}
\theoremstyle{plain} \newtheorem{theorem}[definition]{Theorem}
\theoremstyle{plain} 
\theoremstyle{definition} 
\theoremstyle{plain} 
\theoremstyle{definition} 
\theoremstyle{plain} 
\theoremstyle{plain}

\DeclareMathOperator{\dive}{div}

\DeclareMathOperator{\Lip}{Lip}

\DeclareMathOperator{\Wedge}{{\textstyle\bigwedge}}
\DeclareMathOperator{\spn}{span}

\newcommand{\ee}{\mathrm{e}}

\newcommand{\sbullet}{\begin{picture}(1,1)(-0.5,-2.5)\circle*{2}\end{picture}}
\newcommand{\frarg}{\,\sbullet\,}

\newcommand{\R}{\mathbb{R}}

\newcommand{\N}{\mathbb{N}}

\newcommand{\Id}{\mathrm{id}}

\renewcommand{\L}{\mathscr L}

\newcommand{\dd}{\mathrm{d}}

\newcounter{counter}

\newcommand{\curr}[1]{[\![#1]\!]}

\newcommand{\T}{T}

\newcommand{\weaksto}{\overset{*}{\rightharpoonup}}

\newcommand{\Dscr}{\mathscr{D}}

\renewcommand{\L}{\mathscr L}

\theoremstyle{plain} \newtheorem*{theorem*}{Theorem}
\theoremstyle{plain} 
\theoremstyle{plain} \newtheorem*{mthm*}{Main Theorem}
\theoremstyle{plain} \newtheorem*{conjecture*}{Conjecture}
\theoremstyle{plain} 
\theoremstyle{plain} \newtheorem*{problem*}{Problem}

\newcommand{\bbb}[1]{\llbracket #1 \rrbracket}

\numberwithin{equation}{section}

\usepackage{color}
\usepackage{graphicx}
\usepackage{tikz}
\usepackage{framed}
\definecolor{shadecolor}{rgb}{0.94, 0.97, 1.0}

\usepackage{pict2e}
\makeatletter
\DeclareRobustCommand{\intprod}{%
	\mathbin{\mathpalette\int@prod{(0.1,0)(0.9,0)(0.9,0.8)}}}
\DeclareRobustCommand{\restrict}{%
	\mathbin{\mathpalette\int@prod{(0.1,0.8)(0.1,0)(0.9,0)}}}	
\newcommand{\int@prod}[2]{%
	\begingroup
	\sbox\z@{$\m@th#1+$}%
	\setlength\unitlength{\wd\z@}%
	\begin{picture}(1,1)
		\roundcap
		\polyline#2
	\end{picture}%
	\endgroup
}
\makeatother

\usepackage[normalem]{ulem}

\title[Existence and uniqueness for the Lipschitz transport of currents]{Existence and uniqueness for the transport of currents by Lipschitz vector fields}

\author{Paolo Bonicatto}
\address[P.\ Bonicatto]{Mathematics Institute, University of Warwick,
	Zeeman Building, CV4 7HP Coventry, UK}
\email{Paolo.Bonicatto@warwick.ac.uk}

\author{Giacomo Del Nin}
\address[G.\ Del Nin]{Max Planck Institute for Mathematics in the Sciences, Inselstrasse 22, 04103 Leipzig, Germany}
\email{giacomo.delnin@mis.mpg.de}

\author{Filip Rindler}
\address[F.\ Rindler]{Mathematics Institute, University of Warwick,
	Zeeman Building, CV4 7HP Coventry, UK}
\email{f.rindler@warwick.ac.uk}

\begin{document}
	\begin{abstract} This work establishes the existence and uniqueness of solutions to the initial-value problem for the geometric transport equation
		\[
		\frac{\dd}{\dd t}T_t+\Lcal_b T_t=0
		\]
		in the class of $k$-dimensional integral or normal currents $T_t$ ($t$ being the time variable) under the natural assumption of Lipschitz regularity of the driving vector field $b$. Our argument relies crucially on the notion of decomposability bundle introduced recently by Alberti and Marchese. In the particular case of $0$-currents, this also yields a new proof of the uniqueness for the continuity equation in the class of signed measures.
		\vskip.3truecm
		\noindent 
		
	\end{abstract}
	\maketitle
	\setcounter{tocdepth}{1}
	
	\section{Introduction}
	The purpose of this paper is to establish a well-posedness result for the initial-value problem for the geometric transport equation
	\begin{equation}\label{eq:GTE} \tag{GTE}
		\frac{\dd}{\dd t}T_t+\Lcal_b T_t=0.
	\end{equation}
	Here $T_t$ is a time-parametrised family of normal $k$-currents in $\R^d$, and $b\colon \R^d \to \R^d$ is an (autonomous) Lipschitz vector field. Moreover, $\Lcal_b$ stands for the Lie derivative operator taking a normal $k$-current $T$ to its ``derivative along the flow of $b$'', namely another $k$-current (which a-priori is not necessarily normal, or even of finite mass) given via the (dual) Cartan formula
	\[
	\Lcal_b T:=-b\wedge\partial T-\partial(b\wedge T).
	\]
	Weak solutions to~\eqref{eq:GTE} (see below) can thus be understood as the transport of the currents $T_t$ along the flow of $b$. 
	
	For different dimensionalities $k$, the geometric transport equation occurs in many physical models. In the two extremal cases $k=0$ and $k=d$, the equation reduces respectively to the \textit{continuity equation}
	\begin{equation}\label{eq:intro_continuity}
		\frac{\dd}{\dd t}\mu_t+\dive(b\mu_t)=0
	\end{equation}
	for a family of signed measures $\mu_t$ on $\R^d$, and to the \textit{transport equation}
	\begin{equation}\label{eq:intro_transport}
		\frac{\dd}{\dd t}u_t+b\cdot\nabla u_t=0
	\end{equation}
	for a family of functions $u_t \colon \R^d\to\R$. For $k = 1$ one recovers the equations of dislocation transport~\cite{AA,AT,HudsonRindler22}. Further remarks on applications and explicit reductions to these coordinate PDEs can be found in~\cite{BDR}.
	
	More precisely, we understand the geometric transport equation in the following sense: A time-parametrised family $(T_t)_{t \in (0,1)}$, of normal $k$-currents in $\R^d$ is called a weak solution to~\eqref{eq:GTE} if the integrability condition
	\begin{equation}\label{eq:assumption_integrability_masses}
		\int_0^1 \Mbf(T_t)+\Mbf(\partial T_t) \;\dd t<\infty
	\end{equation}
	and
	\begin{equation}\label{eq:PDE_weak_formulation}
		\int_0^1 \dprb{ T_t,\omega} \, \psi'(t) \;\dd t+\int_0^1 \dprb{ \partial (b\wedge T_t) + b \wedge \partial T_t, \omega} \, \psi(t) \;\dd t=0
	\end{equation}
	hold for all $\psi\in \Crm^1_c((0,1))$ and all $\omega\in \Dscr^k(\R^d)$.
	According to~\cite[Lemma 3.5(i)]{BDR} we can assume that the map $t\mapsto T_t$ is weakly$^*$-continuous. Therefore, it makes sense to consider the initial-value problem 
	\begin{equation}\label{eq:initial_value_problem}
		\left\{\begin{aligned}
			\frac{\dd}{\dd t}T_t+\Lcal_b T_t &=0,  \qquad t \in (0,1),\\
			T_0 &=\overline T,
		\end{aligned}\right.
	\end{equation}
	where the initial condition is understood in the sense that $\overline T$ is the weak$^*$ limit of $T_t$ as $t\to 0$.
	We emphasise that the time interval
	is $[0,1]$ throughout this work for reasons of simplicity only. With minor adaptations, the proofs work also for any (possibly unbounded) time interval.  
	
	The existence and uniqueness theory of the transport equation~\eqref{eq:intro_transport} is by now well-developed, even under weak regularity assumptions on the vector field $b$ (see, e.g.,~\cite{DPL,AmbrosioBV,AmbrosioCrippa,AGS}).
	On the other hand, first well-posedness results for~\eqref{eq:GTE} in the case of a general dimension $k$ have only recently been shown in~\cite{BDR} and only for \emph{smooth} vector fields $b$. 
	In the present work, we will however show that a Lipschitz vector field $b$ is sufficient to ensure existence and uniqueness of the initial-value problem for the geometric transport equation~\eqref{eq:GTE} (see Section~\ref{sc:prelims} for notation). This regularity class is the natural one to consider in the framework of normal or integral currents. 
	
	The main difficulty in extending the results in \cite{BDR} to the Lipschitz setting is the following: Solutions to~\eqref{eq:GTE} can be understood as the transport of the currents $T_t$ along the flow lines of $b$. However, since the flow of $b$ is merely Lipschitz, it may well occur that it is not (fully) differentiable \emph{anywhere} on the support of the transported currents, therefore questioning the possibility of performing pointwise computations (which is the strategy adopted in the smooth case). 
	
	Our main result reads as follows: 
	
	\begin{theorem}\label{thm:main} 
		Let $b\colon \R^d \to \R^d$ ($d \in \N$) be a globally bounded and Lipschitz vector field with flow $\Phi_t = \Phi(t,\frarg) \colon \R^d \to \R^d$ and let $\overline T \in \Nrm_k(\R^d)$ be a $k$-dimensional normal current on $\R^d$. Then, the initial-value problem
		\[
		\left\{\begin{aligned}
			\frac{\dd}{\dd t}T_t+\Lcal_b T_t &=0,  \qquad t \in (0,1),\\
			T_0 &=\overline T
		\end{aligned}\right.
		\]
		admits a solution $(T_t)_{t \in (0,1)} \subset \Nrm_k(\R^d)$ of normal $k$-currents, which is unique in the class of normal $k$-currents. The solution is given by the pushforward of the initial current under the flow, namely, $T_t=(\Phi_t)_*\overline T$. In particular, if $\overline T$ is integral, then so is $T_t$, $t \in (0,1)$.
	\end{theorem}
	
	The proof of this theorem is carried out in Sections~\ref{sc:exist} and~\ref{sc:unique}. It employs the space-time approach introduced in~\cite{HudsonRindler22,Rindler23,BDR}, whose central idea is to consider the $(1+k)$-current in space-time $\R \times \R^d$ whose slices are $\delta_t \times T_t$.
	
	Another interesting feature of our argument is that we rely on a relatively recent tool from Geometric Measure Theory, namely the notion of \textit{decomposability bundle}, introduced by Alberti and Marchese~\cite{AM}. This tool ensures that, while full differentiability of $\Phi_t$ on the support of $\overline T$ may fail, the derivative of $\Phi_t$ still exists in a sufficiently good sense to define the pushforward $(\Phi_t)_*\overline T$ \emph{pointwise} (and not via the homotopy formula). 
	
	Our proof seems to be new even in the case $k = 0$, i.e., for the continuity equation, where well-posedness results in the Lipschitz setting are well-known (see, e.g. \cite{AGS} and also \cite{AB_osgood,BG,clop}). In fact, in this case our proof can be simplified somewhat (in particular, no multilinear algebra is required) and we present the argument separately in the last Section~\ref{sc:continuity}.

	\subsection*{Acknowledgements}
	
	This project has received funding from the European Research Council (ERC) under the European Union's Horizon 2020 research and innovation programme, grant agreement No 757254 (SINGULARITY). The authors thank Elio Marconi for a useful discussion which led to the content of Remark~\ref{rem:no_dec_bundle}. 
	
	\section{Notation and preliminaries} \label{sc:prelims}
	
	This section fixes our notation and recalls some basic facts. We refer the reader to~\cite{Federer69book,KrantzParks08book} for proofs.
	
	\subsection{Linear and multilinear algebra.}
	Let $d \in \N$ be the ambient dimension. We will often use the projection maps $\tbf \colon\R\times \R^d\to \R$, $\pbf \colon\R\times \R^d\to \R^d$ from the (Euclidean) space-time $\R\times\R^d$ onto the time and space variables, respectively, which are given as
	\begin{equation}\label{eq:projections_def}
		\tbf(t,x) = t,  \qquad
		\pbf(t,x) = x.
	\end{equation}
	We also define, for every given $t\in\R$, the immersion map $\iota_t\colon \R^d \to \R\times\R^d$ by
	\[
	\iota_t(x):=(t,x),  \qquad (t,x) \in \R\times \R^d.
	\]
	
	If $V$ is a (finite-dimensional, real) vector space, for every $k \in \N$, we let $\Wedge^k V$ be the space of $k$-covectors on $V$, and $\Wedge_k(V)$ be the space of $k$-vectors on $V$. 
	We denote the duality pairing between $v\in\Wedge_k V$ and $\alpha\in\Wedge^k V$ by $\dprb{ v,\alpha}$. 
	Referring to~\cite[Section 5.8]{AM}, given a $k$-vector $v$ in $V$, we denote by $\spn(v)$ the smallest linear subspace $W$ of $V$ such that $v \in \Wedge_k(W)$. A similar definition is given for a $k$-covector $\alpha$ in $V$. 
	
	Whenever $V$ is an inner product space, we can endow $\Wedge_k V$ with an inner product (Euclidean) norm $|\cdot|$ by declaring $\ee_I:=\ee_{i_1}\wedge\ldots \wedge \ee_{i_k}$, as $I$ varies in the $k$-multi-indices of $\{1,\ldots, n\}$, as \textbf{orthonormal} whenever $\ee_1,\ldots,\ee_n$ are an orthonormal basis of $V$. A simple $k$-vector $\eta\in\Wedge_k V$ is called a \textbf{unit} if there exists an orthonormal family $v_1,\ldots,v_k$ such that $\eta=v_1\wedge\ldots\wedge v_k$, or equivalently if its Euclidean norm $|\eta|$ equals 1. We define the \textbf{comass} of $\alpha\in\Wedge^kV$ as
	\[
	\|\alpha\|:=\sup\bigl\{\dprb{ \eta,\alpha}: \eta\in\Wedge_k V, \text{ simple, unit}\bigr\}
	\]
	and the \textbf{mass} of $\eta\in\Wedge_kV$ as
	\[
	\|\eta\|:=\sup\bigl\{\dprb{ \eta,\alpha}: \alpha\in\Wedge^k V,\,\|\alpha\|\leq 1\bigr\}.
	\]
	We have that $\|\eta\|=|\eta|$ if and only if $\eta$ is simple (and the same holds for covectors). From the definition it follows that
	\begin{equation}\label{eq:cauchy_schwartz_mass_comass}
		\absb{\dprb{ \eta,\alpha}}\leq \,\|\eta\|\|\alpha\|,\qquad\text{for every $\eta\in\Wedge_k V$, $\alpha\in\Wedge^k V$.}
	\end{equation}

	\begin{lemma}\label{lemma:wedge_perp}\phantom{.}
		\begin{enumerate}
			\item[(i)] Let $\tau\in\Wedge_1 \R^n$, $\sigma\in\Wedge_k(\R^n)$ and $\alpha\in\Wedge^1(\R^n)$, $\beta\in\Wedge^k(\R^n)$. Suppose that $\dprb{ v,\alpha}=0$ for every $v\in \mathrm{span}(\sigma)$. Then
			\[
			\dprb{ \tau\wedge\sigma,\alpha\wedge\beta}=\dprb{\tau,\alpha} \dprb{ \sigma,\beta}.
			\]
			\item[(ii)] Let $\tau\in\Wedge_1 \R^n$, $\sigma\in\Wedge_k(\R^n)$ and $\alpha\in\Wedge^{k+1}(\R^n)$. Suppose that $\tau\perp \mathrm{span}(\alpha)$.
			Then,
			\[
			\dprb{ \tau\wedge\sigma,\alpha}=0.
			\]
		\end{enumerate}
	\end{lemma}
	
	\begin{proof}
		\noindent\textit{Ad~(i).} Let $V:=\spn(\tau_2)$, and let $v_1,\ldots,v_m$ be a basis of $V$. We can write $\sigma$ as
		\[
		\sigma=\sum_I \sigma_I v_I
		\]
		where $I=\{i_1,\ldots,i_k\}$ runs among all $k$-multi-indices of $\{1,\ldots,m\}$ and $v_I:=v_{i_1}\wedge\ldots\wedge v_{i_k}$. By linearity, it is sufficient to prove the claim when $\sigma=v_I$ for some $k$-multi-index $I$ (without loss of generality, up to a relabeling we can assume that $I=\{1,\ldots,k\}$) and when $\beta=\beta_1\wedge\ldots\wedge \beta_k$ for some $\beta_1,\ldots,\beta_k\in\R^n$. In this case, setting for notational simplicity $\tau=v_0$ and $\alpha=\beta_0$, we have
		\begin{align*}
			\dprb{\tau\wedge\sigma,\alpha\wedge \beta}& =\dprb{ v_0\wedge v_{1}\wedge\ldots\wedge v_{k},\beta_0\wedge \beta_1\wedge\ldots\wedge\beta_k}\\
			&=\det(\dprb{ v_i,\beta_j})_{i,j=0,\ldots, k}\\
			&=\dprb{ v_0,\beta_0} \det(\dprb{ v_i,\beta_j})_{i,j=1,\ldots, k}\\
			&=\dprb{\tau,\alpha}\dprb{ \sigma,\beta},
		\end{align*}
		where we used Laplace's determinant expansion on the first column ($j=0$) and the assumption that $\dprb{ v_i,\alpha}=0$ for $i=1,\ldots,k$.
		
		\medskip\noindent\textit{Ad~(ii).} The proof is very similar to the proof of~(i).
	\end{proof}

	Given a linear map $S\colon V\to W$, we define $\Wedge^k S \colon \Wedge_k V\to \Wedge_k W$ by
	\[
	( \Wedge^k S) (v_1\wedge \ldots\wedge v_k)= (Sv_1)\wedge\ldots\wedge (Sv_k),
	\]
	on simple vectors and then we extend this definition by linearity to all of $\Wedge_k V$. If there is no risk of confusion, we will often write simply $S$ instead of $\Wedge^{k}S$ to denote the extension of the map $S$ to $\Wedge^k V$.  
	
	We denote by $\Dscr^k(\R^d)$ the space of smooth $k$-forms on $\R^d$ with compact support. The integer $k$ will also be called the \textbf{degree} of $\omega \in \Dscr^k(\R^d)$ and will be written as $\deg{(\omega)}$. 	As usual, $d\omega \in \Dscr^{k+1}(\R^d)$ denotes the \textbf{exterior differential} of a $k$-form $\omega \in \Dscr^k(\R^d)$. 	We recall in particular the following Leibniz formula holds for any pair $\alpha, \beta$ of differential forms:
	\[
	d(\alpha\wedge\beta)=d\alpha \wedge \beta+(-1)^{\deg(\alpha)} \alpha\wedge d\beta.
	\]
	The \textbf{comass of a form} $\omega\in \Dscr^k(\R^d)$ is
	\[
	\|\omega\|_\infty:= \sup_{x\in \R^d} \|\omega(x)\|.
	\]

	The \textbf{pullback of a covector} $\alpha\in\Wedge^kV$ with respect to a linear map $S:V\to W$ is given by
	\[
	\dprb{  v_1\wedge\ldots\wedge v_k,S^* \alpha} := \dprb{  (Sv_1)\wedge\ldots\wedge (Sv_k),\alpha}
	\]
	on simple $k$-vectors and then extended by linearity. Therefore,
	\[
	\dprb{ \eta,S^*\alpha}= \dprb{(\Wedge^k S) \eta,\alpha}.
	\]
	If $f\colon \R^d \to  \R^d$ is differentiable and proper (meaning that preimages of compact sets are themselves compact) and $\omega\in \Dscr^k(\R^d)$, we define the \textbf{pullback} $f^*\omega \in \Dscr^k(\R^d)$ to be the differential form $f^*\omega$ given by
	\[
	\dprb{ v,(f^*\omega)(x)}:= \dprb{ Df(x)[v],\omega(f(x))}, \qquad v\in \Wedge_k(\R^d).
	\]
	The properness of the pullback map $f$ can be omitted if the pullback form $f^* \omega$ is always integrated against a current of compact support. Here, we usually use $\tbf$ and $\pbf$ as pullback maps, which are not proper, but in all instances the compound expressions in which they appear are compactly supported and no issue of well-definedness arises. For instance, in an expression of the form $\dprb{S, \tbf^* \alpha \wedge \pbf^* \beta }$ with $S \in \Dscr_{1+k}(\R \times \R^d)$ and $\alpha \in \Dscr^0((0,1))$, $\beta \in \Dscr^k(\R^d)$ the product $\tbf^* d\alpha \wedge\pbf^* \beta$ is compactly supported in $\R \times \R^d$ even though its two factors $\tbf^* \alpha$, $\pbf^* \beta$ are not. 
	We see from~\eqref{eq:cauchy_schwartz_mass_comass} that for every $k$-vector $v \in \Wedge_k \R^d$,
	\begin{equation}\label{eq:pullback_mass_comass}
		\absb{\dprb{ v,(f^*\omega )(x)}} \leq \|Df(x)[v]\|\cdot\|\omega (f(x))\| , \qquad  x \in \R^d.
	\end{equation}
	
	\subsection{Currents} We refer to~\cite{Federer69book} for a comprehensive treatment of the theory of currents, summarising here only the main notions that we will need. The space of \textbf{$k$-dimensional currents} $\mathscr D_k(\R^d)$ is defined as the dual of $\Dscr^k(\R^d)$, where the latter space is endowed with the locally convex topology induced by local uniform convergence of all derivatives. Then, the notion of \textbf{(sequential weak*) convergence} is the following: 
	\begin{equation*}
		T_n \weaksto T \text{ in the sense of currents } \Longleftrightarrow \dprb{ T_n, \omega} \to \dprb{ T, \omega} \quad  \text{for all $\omega \in \Dscr^k(\R^d)$.}
	\end{equation*}
	
	The \textbf{boundary} of a current is defined as the adjoint of De Rham's differential: if $T$ is a $k$-current, then $\partial T$ is the $(k-1)$-current given by 
	\begin{equation*}
		\dprb{ \partial T, \omega } = \dprb{ T, d\omega },   \qquad  \omega \in \Dscr^{k-1}(\R^d). 
	\end{equation*}
	We denote by $\Mrm_k(\R^d)$ the space of $k$-\textbf{currents with finite mass} in $\R^d$, where the \textbf{mass} of a current $T \in \mathscr{D}_k(\R^d)$ is defined as 
	\begin{equation*}
		\mathbf{M}(T):= \sup \left\{ \dprb{ T,\omega }:  \omega \in \Dscr^k(\R^d), \| \omega\|_\infty \le 1 \right\}.
	\end{equation*}
	
	Let $\mu$ be a finite measure on $\R^d$ and let ${\tau} \colon \R^d \to \Wedge_k(\R^d)$ be a map in ${\rm L}^1(\mu)$. Then we define the current $T:={\tau}\mu$ as 
	\begin{equation*}
		\dprb{ T, \omega } = \int_{\R^d} \dprb{ \tau(x),\omega(x) }\; \dd \mu(x).
	\end{equation*}
	We recall that all currents with finite mass can be represented as $T= \tau \mu$ for a suitable pair $\tau, \mu$ as above. In the case when $\|\tau\|=1$ $\mu$-a.e., we denote $\mu$ by $\|T\|$ and we call it the \textbf{mass measure} of $T$. As a consequence, we can write $T=\vec{T}\|T\|$, where $\|\vec{T}\|=1$ $\|T\|$-almost everywhere. One can check that, if $T=\tau\mu$ with $\tau\in \Lrm^1(\mu)$, then $\|T\|=\|\tau\| \mu$, hence
	\[
	\mathbf{M}(T) = \int_{\R^d} \|\tau(x)\|\; \dd \mu(x).
	\]
	Given a current $T=\tau \mu \in \Dscr_k(\R^d)$ with finite mass and a vector field $v\colon \R^d \to \R^d$ defined $\|T\|$-a.e., we define the \textbf{wedge product}
	\[
	v\wedge T:=(v\wedge \tau) \mu\in \Dscr_{k+1}(\R^d).
	\]
	
	The \textbf{pushforward} of $\T$ with respect to a proper $\Crm^1$-map $f\colon \R^d \to \R^d$ is defined by
	\[
	\dprb{ f_* T,\omega}=\dprb{ \T,f^* \omega}. 
	\]
	In view of~\eqref{eq:pullback_mass_comass} we have the estimate 
	\begin{equation}\label{eq:mass_of_pushforward}
		\Mbf(f_{*} T) \leq \int_{\R^d} \bigl\| Df(x)[\tau(x)] \bigr\| \; \dd \|T\|(x)
		\leq \|Df\|_\infty^k \Mbf(T).
	\end{equation}
	In the case of measures, we employ instead the standard notation $f_{\#} \mu$ for the pushforward of $\mu$ under a map $f$, namely, the measure defined by $f_{\#}\mu (A)=\mu(f^{-1}(A))$.
	
	If $T$ is \textbf{simple}, i.e., $\vec{T}$ is a simple $k$-vector $\|T\|$-almost everywhere, then the same inequality holds with the mass norm $\|\cdot\|$ replaced by the Euclidean norm $|\cdot|$.
	
	Given two currents $T_1\in \mathscr{D}_{k_1}(\R^{d_1})$ and $T_2\in\mathscr{D}_{k_2}(\R^{d_2})$, their \textbf{product} $T_1\times T_2$ is a well-defined current in $\mathscr{D}_{k_1+k_2}(\R^{d_1}\times\R^{d_2})$. The boundary of the product is given by
	\begin{equation}\label{eq:boundary_of_product}
		\partial (T_1\times T_2)=\partial T_1 \times T_2 + (-1)^{k_1} T_1\times \partial T_2.
	\end{equation}
	A $k$-current on $\R^d$ is said to be \textbf{normal} if both $T$ and $\partial T$ have finite mass. The space of normal $k$-currents is denoted by $\mathrm N_k(\R^d)$.	
	
	\subsection{Decomposability bundle}
	We recall from~\cite{AM} the definition and a few basic facts about the decomposability bundle. Given a measure $\mu$ on $\R^n$, the \textit{decomposability bundle} is a $\mu$-measurable map $x\mapsto V(\mu,x)$ (defined up to $\mu$-negligible sets) which associates to $\mu$-a.e.\ $x$ a subspace $V(\mu,x)$ of $\R^n$. We refer to~\cite{AM} for the precise notion of $\mu$-measurability of subspace-valued maps.
	
	The map  $V$ satisfies the following property: Every Lipschitz function $f \colon \R^n\to\R$ is differentiable at $x$ along the subspace $V(\mu,x)$, for $\mu$-a.e.\ $x\in\R^n$. Moreover, this map is $\mu$-maximal in a suitable sense, meaning that $V(\mu,x)$ is, for $\mu$-a.e.\ $x$, the biggest subspace with this property (see~\cite[Theorem~1.1]{AM}). 
	The directional derivative of $f$ at $x$ in direction $v \in V(\mu,x)$ will be denoted by $Df(x)[v]$. Observe that this is a slight abuse of notation, as the full differential $Df$ might not exist at $x$, even though the directional derivative exists.
	
	A key fact about the decomposability bundle with regard to the theory of normal currents is the following~\cite[Theorem~5.10]{AM}: Given a normal $k$-current $T=\vec{T}\|T\|$ in $\R^n$, it holds that
	\begin{equation}\label{eq:span_in_bundle}
		\spn(\vec{T})\subseteq V(\|T\|,x)\qquad\text{for $\|T\|$-a.e.\ $x\in\R^n$}.
	\end{equation}
	In particular, given any Lipschitz function $f$, we can define $D_T f$ at $\|T\|$-a.e.\ point as the restriction of the differential of $f$ to $\spn(\vec T)$. We will usually just write $Df$ instead of $D_T f$ when this differential is evaluated in a direction in $\spn(\vec T)$.
	
	Recall that for a normal current $T\in \Nrm_k(\R^d)$ it is possible to define the pushforward $f_*T$ when $f\colon \R^d \to \R^d$ is merely Lipschitz via the homotopy formula~\cite[4.1.14]{Federer69book}. Classically, no explicit formula for this pushforward was available. However, it is shown in~\cite[Proposition~5.17]{AM} that the pushforward formula, in fact, remains true:
	
	\begin{lemma}\label{lemma:pushforward_currents}
		Suppose that $T=\tau\mu$ is a normal $k$-current in $\R^n$, and $f \colon \R^n\to \R^m$ is a proper, injective Lipschitz map. Then, the pushforward current $f_*T$ satisfies
		\[
		f_*T= \tilde\tau\tilde\mu,
		\]
		where $\tilde\mu =f_\#\mu $, and $\tilde \tau(y)=Df(x) [\tau(x)] = D_T f(x) [\tau(x)]$ with $y=f(x)$.
	\end{lemma}

	\subsection{Flows of Lipschitz fields}\label{ss:flows}
	
	Suppose that $b\colon \R^d \to \R^d$ is a Lipschitz, globally bounded vector field and and let $\Phi \colon \R \times \R^d \to \R^d$ be the associated flow, i.e. the unique map satisfying 
	\begin{equation}
		\left\{\begin{aligned}
			\frac{\partial}{\partial t} \Phi(t,x) &= b(\Phi(t,x)), &&(t,x) \in \R \times \R^d, \\ 
			\Phi(0,x)&=x, && x \in \R^d. 
		\end{aligned}\right.
	\end{equation}
	The existence and uniqueness of the flow map (which is indeed defined on the whole $\R$) follow from the classical Cauchy--Lipschitz theory. This also yields that, for each fixed $t \in \R$, the map  
	\[
	\Phi_t \colon x \mapsto \Phi_t(x):=\Phi(t,x)
	\]
	is Lipschitz. As a consequence of the uniqueness of the flow, we immediately deduce the \emph{semigroup law} formula
	\[
	\Phi_t \circ \Phi_s = \Phi_{t+s}, \qquad  t,s\in \R. 
	\]
	Moreover, for every $t \in \R$ the map $\Phi_t$ is invertible and it holds 
	\[
	(\Phi_t)^{-1} = \Phi_{-t}
	\]
	In the following, we will often need to consider the map $\Psi \colon \R \times \R^d \to \R \times \R^d$ given by 
	\begin{equation}\label{eq:def_Psi}
		\Psi(t,x):=(t,\Phi_t(x)), \qquad (t,x)\in\R\times  \R^d. 
	\end{equation}
	Observe that
	\[
	\Psi^{-1}(s,y)=(s,\Phi_s^{-1}(y))=(s, \Phi_{-s}(y)), \qquad (s, y) \in \R \times \R^d. 
	\]

	We collect in the next lemma some formulae involving the directional derivatives of the maps $\Phi_t$ and $\Psi$. 
	
	\begin{lemma}\label{lemma:directional_derivative_flow} The following assertions hold true: 
		\begin{enumerate} 
			\item[(i)] $D\Phi_t(x)[b(x)]=b(\Phi_t(x))$ for every $(t,x)\in\R\times\R^d$; 
			\item[(ii)] $D\Psi^{-1}(s,y)[(1,b(y))] = (1,0)$  for every $(s,y)\in\R\times\R^d$;
			\item[(iii)] $D\Psi (t,x)[(1,0)]=(1,b(\Phi_t(x)))$ for every $(t,x)\in\R\times\R^d$.
		\end{enumerate} 
	\end{lemma}
	
	Here, in vectors like $(1,0)$ the ``$0$'' is understood as the zero vector in $\R^d$ and thus $(1,0) \in \R \times \R^d$.
	
	\begin{proof}
		We only show~(i) as the proofs of~(ii) and~(iii) are very similar and therefore omitted. By definition, for every fixed $(t,x) \in \R \times \R^d$, we have
		\begin{align*}
			D\Phi_t(x)[b(x)] & = \lim_{h \to 0} \frac{\Phi_t(x+hb(x))-\Phi_t(x))}{h} \\
			& = \lim_{h \to 0} \left[ \frac{\Phi_t(x+hb(x))-\Phi_{t+h}(x)}{h}+ \frac{\Phi_{t+h}(x)-\Phi_t(x)}{h}\right] \\ 
			& = \ell + \frac{\dd}{\dd t} \Phi_t(x)\\ 
			& = \ell + b(\Phi_t(x)), 
		\end{align*}
		so the claim will be proven if we show $\ell = 0$. 
		We observe that
		\begin{equation*}
			\begin{split}
				|\Phi_h(x)-x- hb(x)| & = \left| \int_{0}^h b(\Phi_\tau(x)) - b(x) \;\dd \tau  \right|\\ 
				& \le \int_{0}^h |b(\Phi_\tau(x)) - b(\Phi_0(x))| \;\dd\tau  \\
				& \le \Lip(b) h \sup_{0\leq \tau \leq h}|\Phi_\tau(x)-\Phi_0(x)| \\ 
				& \le \Lip(b) h \sup_{0\leq \tau \leq h}\left|\int_{0}^\tau b(\Phi_\sigma(x)) \;\dd \sigma \right|,
			\end{split}
		\end{equation*}
		whence we obtain the following quadratic error estimate
		\begin{equation}\label{eq:conto_fondamentale}
			|\Phi_h(x)-x- hb(x)| \le\Lip(b) \| b\|_{\infty}  h^2.
		\end{equation}
		Combining this with the semigroup law $\Phi_{t+h}(x) = \Phi_t(\Phi_h(x))$, we can conclude that the following estimate holds: 
		\begin{align*}
			|\Phi_t(x+hb(x))-\Phi_{t+h}(x)| & = |\Phi_t(x+hb(x))-\Phi_{t}(\Phi_h(x))| \\ 
			& \le \Lip(\Phi_t) |x+hb(x)-\Phi_{h}(x)| \\ 
			& \le \Lip(\Phi_t)\Lip(b) \|b\|_\infty h^2. 
		\end{align*}
		As a consequence, 
		\[
		|\ell| = \lim_{h \to 0} \frac{|\Phi_t(x+hb(x))-\Phi_{t+h}(x)|}{|h|} = 0 
		\]
		and the claim is proven.
	\end{proof}
	
	\section{Existence of solutions} \label{sc:exist}
	
	In this section we will prove the existence part of Theorem~\ref{thm:main}.
	
	\begin{proposition}[Existence]\label{prop:existence_currents}
		Let $b\colon \R^d \to \R^d$ be a globally bounded and Lipschitz vector field, and let $\Phi$ be the corresponding flow. Moreover, let $\overline T\in \Nrm_k(\R^d)$. Then, the currents
		\[
		T_t:=(\Phi_t)_* \overline T,  \qquad t \in (0,1),
		\]
		satisfy
		\[
		\left\{\begin{aligned}
			\frac{\dd}{\dd t}T_t+\Lcal_b T_t &=0,  \qquad t \in (0,1), \\
			T_0 &=\overline T.
		\end{aligned}\right.
		\]
	\end{proposition}
	
	\begin{remark}
		Note that the explicit formula for the solution in particular implies that if $\overline T$ is integral, then so is $T_t$ for all $t \in (0,1)$.  
	\end{remark}
	
	\begin{proof}
		\medskip\noindent\textit{Step 1.} 
		We write 
		\[
		\overline{T}  =\tau \mu
		\]
		and 
		\[
		\partial \overline{T}  = \sigma \nu, 
		\]
		where $\mu,\nu$ are non-negative measures and $\tau$ is a unit $k$-vector field (defined $\mu$-a.e.) and $\sigma$ is a unit $(k-1)$-vector field (defined $\nu$-a.e.). 
		We also define the following \emph{cylindrical current}  
		\[
		C := \bbb{0,1} \times \overline{T} \in \Nrm_{k+1}(\R\times \R^d).
		\]
		Via the formula for the boundary of the product we observe that 
		\[
		\partial C \restrict ((0,1)\times \R^d) = -\bbb{0,1} \times \partial \overline{T}.
		\] 
		The current $C$ can be written in the form 
		\[
		C=(\ee_0 \wedge \tau) \L^1 \otimes \mu,
		\]
		where we have identified $\tau$ with $(\iota_t)_* \tau$ (recall that $\iota_t(x):=(t,x)$).
		Similarly, the current $\partial C$ can be written in the form 
		\[
		\partial C=-(\ee_0 \wedge \sigma )\L^1 \otimes \nu.
		\]
		By Lemma~\ref{lemma:pushforward_currents} we have
		\[
		(\Phi_t)_*\overline{T}=  \tau_t \mu_t,
		\]
		where 
		\[
		\tau_t(y):=D\Phi_t(\Phi_{-t}(y))[\tau(\Phi_{-t}(y))]
		\qquad\text{and}\qquad
		\mu_t := (\Phi_t)_\#\mu.
		\]
		We also let $Z:=\Psi_{*}C$, where $\Psi$ is the map defined in~\eqref{eq:def_Psi} (this current is the candidate to be the space-time solution). 
		By Lemma~\ref{lemma:pushforward_currents} we obtain that the current $Z$ can be written as $Z=z\lambda$, where 
		\[
		\lambda := \Psi_{\#} (\L^1\otimes \mu) =  (\Id,\Phi_t)_{\#} (\L^1 \otimes \mu) =  \L^1 (\dd t) \otimes (\Phi_t)_{\#} \mu =\L^1 (\dd t) \otimes \mu_t
		\]
		and
		\begin{align*}
			z(s,y) &= D\Psi(\Psi^{-1}(s,y)) \bigl[(\ee_0 \wedge (\iota_s)_*\tau)(\Psi^{-1}(s,y)) \bigr] \\
			&= D\Psi(\Psi^{-1}(s,y))[\ee_0] \wedge D\Psi(\Psi^{-1}(s,y))\big[ (\iota_s)_*[\tau(\Phi_s^{-1}(y))]\big]\\
			& = (1,b(y))\wedge D\Psi(\Psi^{-1}(s,y))\big[ (\iota_s)_*[\tau(\Phi_s^{-1}(y))]\big]\\
			& =(1,b(y))\wedge (\iota_s)_*[(\Phi_s)_*\tau](s,y),
		\end{align*}
		where we have used Lemma~\ref{lemma:directional_derivative_flow}~(iii) to infer that 
		\[
		D\Psi(\Psi^{-1}(s,y))[\ee_0] = (1,b(y))
		\]
		and the identity $\Psi\circ\iota_s=\iota_s\circ\Phi_s$.

		\medskip\noindent\textit{Step 2.} 
		We are now going to test $\partial Z$ against a generic form
		\[
		\eta=\tbf^*\alpha\wedge \pbf^*\beta
		\]
		with $\alpha\in\Dscr^0((0,1))$ and $\beta\in\Dscr^k(\R^d)$, writing it in two ways, namely as $\dpr{ Z,d\eta}$ and as $\dpr{ \partial Z,\eta}$, and then equating the two expressions.
		
		First, by the Leibniz rule we have that $d\eta=\tbf^*d\alpha\wedge\pbf^*\beta+\tbf^*\alpha\wedge\pbf^*d\beta$. When we test against the first term, recalling also Lemma~\ref{lemma:wedge_perp}, we obtain
		\begin{align*}
			\dprb{ Z,\tbf^*d\alpha\wedge\pbf^*\beta } &= \int_0^1\int_{\R^d} \dprb{(1,b)\wedge (\iota_s)_*[(\Phi_s)_*\tau],\tbf^*d\alpha\wedge\pbf^*\beta} \;\dd\mu_s\;\dd s\\
			&= \int_0^1\int_{\R^d} \dprb{ (1,b),\tbf^*d\alpha} \dprb{ (\Phi_s)_*\tau,\beta} \;\dd\mu_s \;\dd s\\
			&= \int_0^1 \alpha'(s) \, \dprb{ T_s,\beta}\;\dd s.
		\end{align*}
		When we test against the second term we obtain
		\begin{align*}
			\dprb{ Z,\tbf^*\alpha\wedge\pbf^*d\beta } &= \int_0^1\int_{\R^d} \alpha(s) \, \dprb{(1,b)\wedge (\iota_s)_*[(\Phi_s)_*\tau],\pbf^*d\beta} \;\dd\mu_s\;\dd s\\
			&= \int_0^1\int_{\R^d}\alpha(s) \, \dprb{ b\wedge (\Phi_s)_*\tau,d\beta} \;\dd\mu_s\;\dd s\\
			&= \int_0^1 \alpha(s) \, \dprb{ b\wedge T_s,d\beta}\;\dd s.
		\end{align*}
		Combining, we arrive at
		\[
		\dprb{Z,d\eta} = \int_0^1 \alpha'(s) \, \dprb{ T_s,\beta} + \alpha(s) \, \dprb{ b\wedge T_s,d\beta} \;\dd s.
		\]
		
		Second, using that
		\[
		\partial Z \restrict ((0,1)\times \R^d)
		=\Psi_*(\partial C \restrict ((0,1)\times \R^d))=\Psi_*(-\curr{0,1}\times \partial \overline{T}),
		\]
		we can rely on Lemma~\ref{lemma:pushforward_currents} to write
		\[
		\partial Z \restrict ((0,1)\times \R^d)
		= \xi \,\Psi_\# \|\partial C \restrict ((0,1)\times \R^d)\|
		=\xi\, (\L^1 (\dd s) \restrict (0,1)) \otimes \nu_s,
		\]
		where $\nu_s =(\Phi_s)_\#\|\partial\overline{T}\|$ and
		\[
		\xi(s,y)=\Psi_* (-\ee_0\wedge \sigma)(s,y)=-(1,b(y))\wedge (\Psi_*(\iota_s)_*\sigma) (s,y).
		\]
		Here we again used that $(\Psi_*\ee_0)(s,y)=(1,b(y))$ by Lemma~\ref{lemma:directional_derivative_flow}. Therefore,
		\begin{align*}
			\dprb{ \partial Z,\eta} &=-\int_0^1\int_{\R^d} \dprb{ (1,b)\wedge (\Psi_*(\iota_s)_*\sigma) (s,y),(\tbf^*\alpha\wedge \pbf^*\beta )(s,y)} \;\dd\nu_s(y)\;\dd s \\
			&= -\int_0^1\int_{\R^d} \alpha(s) \, \dprb{ b(y)\wedge \pbf_*\Psi_*(\iota_s)_*\sigma (y),\beta(y)} \;\dd\nu_s(y)\;\dd s \\
			&= -\int_0^1\int_{\R^d} \alpha(s) \, \dprb{ b(y)\wedge (\Phi_s)_*\sigma (y),\beta(y)} \;\dd\nu_s(y)\;\dd s \\
			&= -\int_0^1\alpha(s) \, \dprb{ b\wedge\partial T_s, \beta}\;\dd s.
		\end{align*}
		
		\medskip\noindent\textit{Step 3.} From Step 2 we obtain that
		\begin{align*}
			-\int_0^1\alpha(s) \, \dprb{ b\wedge\partial T_s, \beta}\;\dd s 
			&=\dprb{ \partial Z,\eta} \\
			&=\dprb{  Z,d\eta} \\
			&= \int_0^1 \alpha'(s) \, \dprb{ T_s,\beta} + \alpha(s) \, \dprb{ b\wedge T_s,d\beta}\;\dd s
		\end{align*}
		for every $\alpha\in\Dscr^0((0,1))=\Crm_c^\infty((0,1))$ and every $\beta\in\Dscr^k(\R^d)$.
		Rearranging terms gives exactly the weak formulation~\eqref{eq:PDE_weak_formulation} of~\eqref{eq:GTE}.
	\end{proof}

	\begin{remark} Proposition \ref{prop:existence_currents} can also be proved by means of a simple approximation argument that does not necessitate the decomposability bundle (but assumes the existence of solutions when  $b$ is smooth). Indeed, if $(b^\eps)_\eps$ denotes a family of smooth vector fields approximating $b$ in the uniform norm and with equibounded Lipschitz constants, one can consider the pushforwards $T_t^\eps := (\Phi^\eps_t)_{*}\overline T$, where $\Phi^\eps$ denotes the flow of $b^\eps$. This family of currents satisfies the equation with $b^\eps$ by \cite[Theorem~3.6]{BDR} and is equibounded in mass for $\eps>0$ by \eqref{eq:mass_of_pushforward}. Therefore any limit point as $\eps\to 0$ satisfies \eqref{eq:GTE} by linearity of the equation. However, we presented an argument based on the decomposability bundle because the same tool will be used below in the proof of uniqueness.
	\end{remark}

	\section{Uniqueness of solutions} \label{sc:unique}
	
	We now turn our attention to the uniqueness part of Theorem~\ref{thm:main}. As it will be shown in the last step of the proof of Proposition~\ref{prop:uniqueness_currents} below, it is enough to prove uniqueness under the assumption that $\partial T_t=0$ for every $t \in [0,1)$ by an elementary argument. 
	
	Let thus $(T_t)_{t \in (0,1)} \subset \Nrm_k(\R^d)$ with $\partial T_t=0$ ($t \in[0,1)$) be a weakly$^*$-continuous solution to~\eqref{eq:GTE} (see~\cite[Lemma 3.5(i)]{BDR} for why we may assume weak*-continuity). We decompose $T_t=\vec{T}_t\|T_t\|$, with $\vec{T}_t$ unit $k$-vectors.
	Let $\vec{T}:(0,1)\times\R^d\to \Wedge_k(\R\times\R^d)$ be the $k$-vector field defined $\L^1_t \otimes\|T_t\|$-almost everywhere by
	\[
	\vec{T}(t,x):=(\iota_t)_* \vec{T}_t(x),
	\]
	where we recall that $\iota_t(x):=(t,x)$.
	We define the current 
	\begin{equation}\label{eq:def_U}
		U:=[(1,b(x))\wedge \vec{T}(t,x)] \, \L^1(\dd t) \otimes \|T_t\|(\dd x).
	\end{equation} 
	The following result is an extension to the case of normal currents of~\cite[Prop. 6.6]{BDR}, which covers only the case when the $T_t$'s are integral.
	
	\begin{lemma}\label{lemma:space-time-normal}
		$U$ is a normal $(k+1)$-current in $(0,1)\times\R^d$, and $\partial U \restrict (0,1)\times\R^d = 0$.
	\end{lemma}

	\begin{proof}
		It is easy to see that $\Mbf(U)\leq C \|b\|_\infty \int_0^1 \Mbf(T_t) \;\dd t<\infty$. By the density of linear combinations of tensor forms (see, e.g.,~\cite[4.1.8]{Federer69book}), it is enough to compute the boundary against every $k$-form $\eta=\tbf^* \alpha \wedge \pbf^*\beta$, where either $\alpha\in \Dscr^0((0,1))$ and $\beta\in\Dscr^k(\R^d)$, or $\alpha\in \Dscr^1((0,1))$ and $\beta\in\Dscr^{k-1}(\R^d)$. By the Leibniz rule we have
		\[
		d\eta=\tbf^*d\alpha \wedge \pbf^*\beta+(-1)^{\deg(\alpha)}\tbf^*\alpha\wedge\pbf^*d\beta.
		\]
		
		\medskip\noindent\textit{First case.} Let $\alpha\in \Dscr^0((0,1))$ and $\beta\in\Dscr^k(\R^d)$. Then,
		\begin{align*}
			\dprb{ U,\tbf^*d\alpha \wedge \pbf^*\beta} &= \int_0^1 \int_{\R^d} \dprb{ (1,b(x))\wedge \vec{T}(t,x),\tbf^*d\alpha \wedge \pbf^*\beta (t,x)} \;\dd\|T_t\|(x)\;\dd t\\
			&= \int_0^1 \int_{\R^d} \alpha'(t) \, \dprb{  \vec{T}_t(x), \beta(x) } \;\dd\|T_t\|(x)\;\dd t\\
			&= \int_0^1 \alpha'(t) \, \dprb{  T_t, \beta } \;\dd t,
		\end{align*}
		where we have invoked Lemma~\ref{lemma:wedge_perp}, using that $\mathrm{span}(\vec{T}(t,x))\subset \{0\}\times\R^d$ is space-like and $\tbf^*d\alpha=\alpha'(t)dt$ is time-like.
		On the other hand,
		\begin{align*}
			\dprb{ U,\tbf^*\alpha \wedge \pbf^*d\beta} &= \int_0^1 \int_{\R^d} \dprb{ (1,b(x))\wedge \vec{T}(t,x),\tbf^*\alpha \wedge \pbf^*d\beta (t,x)} \;\dd\|T_t\|(x)\;\dd t\\
			&= \int_0^1 \int_{\R^d} \alpha(t) \, \dprb{ (1,b(x))\wedge \vec{T}(t,x),\pbf^*d\beta (t,x)} \;\dd\|T_t\|(x)\;\dd t\\
			&= \int_0^1 \int_{\R^d} \alpha(t) \, \dprb{ b(x)\wedge \vec{T}_t(x),d\beta (x)} \;\dd\|T_t\|(x)\;\dd t\\
			&=\int_0^1 \alpha(t) \, \dprb{ \partial(b\wedge T_t),\beta}\;\dd t.
		\end{align*}
		Using the weak formulation~\eqref{eq:PDE_weak_formulation} we deduce that $\dprb{ U,d\eta}=0$.
		
		\medskip\noindent\textit{Second case.} Let $\alpha\in \Dscr^1((0,1))$ and $\beta\in\Dscr^{k-1}(\R^d)$. Then, $d\eta=-\tbf^*\alpha\wedge\pbf^*d\beta$ since $d\alpha=0$. Therefore,
		\begin{align*}
			\dprb{ U,\tbf^*\alpha \wedge \pbf^*d\beta} &= -\int_0^1 \int_{\R^d} \dprb{ (1,b(x))\wedge \vec{T}(t,x),\tbf^*\alpha \wedge \pbf^*d\beta (t,x)} \;\dd\|T_t\|(x)\;\dd t\\
			&= -\int_0^1 \int_{\R^d} \dprb{ (1,b(x)),\tbf^*\alpha}\dprb{ \vec{T}(t,x),\pbf^*d\beta (t,x)} \;\dd\|T_t\|(x)\;\dd t\\
			&= -\int_0^1 \int_{\R^d} \alpha(t) \, \dprb{ \vec{T}_t(x),d\beta (t,x)} \;\dd\|T_t\|(x)\;\dd t\\
			&=-\int_0^1 \alpha(t) \, \dprb{\partial T_t,\beta}\;\dd t\\
			&=0.
		\end{align*}
		Passing from the first to the second line we have used Lemma~\ref{lemma:wedge_perp}~(i), while the last equality follows from $\partial T_t=0$.
		
		Using~\eqref{eq:assumption_integrability_masses} and the density of tensor-type forms, we deduce that $\partial U \restrict (0,1)\times\R^d =0$.
	\end{proof}
		
	Next we show a (partial) differentiability property via the Alberti--Marchese theory.
	
	\begin{lemma} Define the measure
		\[
		\mu:=\L^1(\dd t) \otimes \|T_t\|\in\Mcal(\R\times \R^d).
		\]
		Then, $\spn((1,b)\wedge \vec{T_t}(x))\subset V(\mu,(t,x))$ for $\mu$-a.e.\ $(t,x)$ in $(0,1)\times\R^d$.
	\end{lemma}
	
	\begin{proof}
		By Lemma~\ref{lemma:space-time-normal} we have that $U$ is a normal $(k+1)$-current in $(0,1)\times\R^d$. By~\eqref{eq:span_in_bundle} we know that for $\|U\|$-a.e.\ $(t,x)$ in $(0,1)\times\R^d$ it holds that
		\[
		\spn((1,b)\wedge \vec{T_t}(x))=\spn(\vec{U}(t,x))\subset V(\|U\|,(t,x))
		\]
		Moreover, we also have that $\mu\ll\|U\|\ll \mu$ because $1\leq |(1,b)|\leq 1+\|b\|_\infty$ and also $1\leq |(1,b)\wedge \vec{T}|\leq (1+\|b\|_\infty)$. Therefore, $V(\|U\|,(t,x))=V(\mu,(t,x))$ for $\mu$-a.e.\ $(t,x)$ in $(0,1)\times\R^d$, and the conclusion follows.
	\end{proof}
	
	We are now ready to prove the uniqueness property for~\eqref{eq:initial_value_problem}.
	
	\begin{proposition}[Uniqueness]\label{prop:uniqueness_currents} 
		Let $b\colon \R^d \to \R^d$ be a globally bounded and Lipschitz vector field, and let $\Phi$ be the corresponding flow. Moreover, let $\overline T\in \Nrm_k(\R^d)$. Then, a solution to
		\[
		\left\{\begin{aligned}
			\frac{\dd}{\dd t}T_t+\Lcal_b T_t &=0,  \qquad t \in (0,1),\\
			T_0 &=\overline T
		\end{aligned}\right.
		\]
		is unique in the class of normal currents. Moreover, the solution is given by $T_t=(\Phi_t)_*\overline T$.
	\end{proposition}

	\begin{proof}
		We will first prove the proposition under the assumption that $\partial T_t=0$ for $t \in [0,1)$. In Step 4 we will show how to prove the general case.
		
		\medskip\noindent\textit{Step 1.} Let $T_t$ be a solution to~\eqref{eq:initial_value_problem} with $\partial T_t=0$, and define $U$ as in~\eqref{eq:def_U}. We write $U=u\mu$, where $\mu=\L^1_t \otimes \|T_t\|$ and $u(t,x)=(1,b(x))\wedge \vec{T}(t,x)$.
		We now define the $(k+1)$-current 
		\begin{equation}\label{eq:def_W}
			W := (\Psi^{-1})_{*}U. 
		\end{equation}
		Notice that $\Psi^{-1}$ is Lipschitz, so $W$ is well defined (via approximation and the homotopy formula).
		By Lemma~\ref{lemma:pushforward_currents} (applied with $f=\Psi^{-1}$) we can write $W=w\nu$, with $\nu=(\Psi^{-1})_\# \mu$ and
		\begin{align*}
			w(t,x)& = D\Psi^{-1}({\Psi(t,x)}) [u(t,\Phi_t(x))]\\
			& = D\Psi^{-1}({\Psi(t,x)}) [(1,b(\Phi_t(x)))\wedge\vec{T}(t,\Phi_t(x))].
		\end{align*}
		By definition of $D\Psi^{-1}$ and Lemma~\ref{lemma:directional_derivative_flow} (applied with $s=t$ and $y=\Phi_t(x)$), we obtain
		\begin{align}
			w(t,x) 
			& = D\Psi^{-1}({\Psi(t,x)})[(1,b(\Phi_t(x)))] \wedge D\Psi^{-1}({\Psi(t,x)})[\vec{T}(t,\Phi_t(x))]\nonumber\\
			&= (1,0)\wedge D\Psi^{-1}({\Psi(t,x)})[\vec{T}(t,\Phi_t(x))] \nonumber\\
			&=:(1,0)\wedge \tau(t,x).\label{eq:vertical_W}
		\end{align}

		\medskip\noindent\textit{Step 2.}
		Recalling Lemma~\ref{lemma:space-time-normal}, we  have $\partial W=0$ in $(0,1)\times \R^d$. Therefore, for every $\alpha\in\Dscr^0(\R)$ and $\beta\in\Dscr^k(\R^d)$ we have
		\[
		0 = \dprb{ \partial W, \tbf^*\alpha\wedge \pbf^*\beta } = \dprb{ W,\tbf^*d\alpha\wedge\pbf^*\beta+\tbf^*\alpha\wedge\pbf^*d\beta}.
		\]
		As for the second term, we claim that $\dprb{ W, \tbf^* \alpha \wedge\pbf^*d\beta } = 0$.
		Indeed, $\pbf^*d\beta$ only contains terms not involving $dt$. In other words, $\mathrm{span}\,(\pbf^*d\beta)\perp (1,0)$. Therefore, we can apply Lemma~\ref{lemma:wedge_perp}~(ii), taking into account~\eqref{eq:vertical_W}, to obtain the claim.
		
		As for the first term, we will show that
		\[
		\dprb{ W, \tbf^* d\alpha \wedge\pbf^* \beta }=\int_0^1 \alpha'(t) \, \dprb{ (\Phi_{-t})_* T_t,\beta} \;\dd t.
		\]
		This can be proved in the following way: We first apply Lemma~\ref{lemma:wedge_perp}~(i) to deduce that
		\begin{align*}
			\dprb{ w(t,x), \tbf^* d\alpha(t,x) \wedge\pbf^* \beta(t,x)}& =\dprb{ (1,0),\tbf^*d\alpha(t,x)}\dprb{ \tau(t,x),\pbf^* \beta(t,x)}\\
			& =\alpha'(t) \, \dprb{ \tau(t,x),\pbf^* \beta(t,x)}\\
			&= \alpha'(t) \, \dprb{ D\Psi^{-1}({\Psi(t,x)})[\vec{T}(t,\Phi_t(x))],\pbf^* \beta(t,x)} \\
			&=\alpha'(t) \, \dprb{ (\Psi^{-1})_*[\vec{T}(t,\Phi_t(x))],\pbf^* \beta(t,x)}
		\end{align*}
		Recalling that $\nu=(\Psi^{-1})_\#\mu$, and using~\eqref{eq:def_Psi}, we also obtain that $\nu=\L^1(\dd t) \otimes (\Phi_{-t})_\#\|T_t\|$.
		
		We next observe that the maps $\pbf\circ \Psi^{-1}$ and $\Phi_{-t}\circ \pbf$ coincide on $\{t\}\times \R^d$. Accordingly, given any $k$-vector $v\in \Wedge_k(\{0\}\times\R^d)$, we have
		\[
		(\pbf\circ \Psi^{-1})_*v=(\Phi_{-t}\circ \pbf)_*v.
		\]
		In particular, we can apply the previous equality with $v=\vec{T}(t,\Phi_t(x))$ to obtain
		\begin{align*}
			(\pbf\circ\Psi^{-1})_*[\vec{T}(t,\Phi_t(x))]&=  (\Phi_{-t}\circ\pbf)_*[\vec{T}(t,\Phi_t(x))]\\
			&=  (\Phi_{-t})_*\pbf_*[\vec{T}(t,\Phi_t(x))]\\
			&=(\Phi_{-t})_*[\vec{T}_t(\Phi_t(x))],
		\end{align*}
		and therefore
		\begin{align*}
			\dprb{ W, \tbf^* d\alpha \wedge\pbf^* \beta }
			& = \int_{(0,1)\times\R^d} \dprb{ w,\tbf^* d\alpha \wedge\pbf^* \beta } \;\dd\nu \\
			&= \int_0^1 \alpha'(t)\int_{\R^d}  \dprb{ (\Psi^{-1})_*[\vec{T}(t,\Phi_t(x))],\pbf^* \beta(t,x)}\;\dd(\Phi_{-t})_\#\|T_t\|(x)\;\dd t\\
			&= \int_0^1 \alpha'(t)\int_{\R^d}  \dprb{ (\Phi_{-t})_*[\vec{T}_t(\Phi_t(x))],\beta(x)}\;\dd(\Phi_{-t})_\#\|T_t\|(x)\;\dd t\\
			&= \int_0^1 \alpha'(t) \, \dprb{ (\Phi_{-t})_* T_t,\beta} \;\dd t,
		\end{align*}
		where we used the explicit formula for the pushforward of currents given by Lemma~\ref{lemma:pushforward_currents}.

		\medskip\noindent\textit{Step 3.}
		By Step 1 and Step 2 we have shown that 
		\[
		\int_0^1 \alpha'(t) \, \dprb{ (\Phi_{-t})_* T_t,\beta} \;\dd t=0\qquad\text{for every $\alpha\in\Dscr^0((0,1))$ and $\beta\in\Dscr^k(\R^d)$.}
		\]
		This implies that, for any fixed $\beta\in\Dscr^k(\R^d)$, the map $t\mapsto \dprb{ (\Phi_{-t})_* T_t,\beta} $ is constant, and in particular equal to its value at 0 (in this context recall that we have chosen the weakly$^*$-continuous representative of $t\mapsto T_t$). We conclude that $(\Phi_{-t})_* T_t = \overline T$, or equivalently, $T_t=(\Phi_t)_* \overline T$, which is the conclusion of the proposition in the case $\partial T_t = 0$ ($t \geq 0$).

		\medskip\noindent\textit{Step 4.} Assume now that $(T_t)_{t \in (0,1)}$ is a solution of~\eqref{eq:initial_value_problem} with $T_t$ normal but not necessarily boundaryless. Testing~\eqref{eq:PDE_weak_formulation} with exact forms $\omega=d\eta$, we see that the boundaries $(\partial T_t)_{t\in (0,1)}$ solve the geometric transport equation with the same driving vector field $b$ and initial datum $\partial \overline{T}$. Since the $\partial T_t$ are boundaryless, by Step 3 we obtain that 
		\[
		\partial T_t=(\Phi_t)_* (\partial\overline{T})=\partial ((\Phi_t)_*\overline{T}).
		\]
		Consider now $S_t:=T_t-(\Phi_t)_* \overline{T}$. The $S_t$ are normal currents with $\partial S_t=0$, and by the linearity of~\eqref{eq:GTE}, they still solve the same equation. Since $S_0=0$, again by Step 3 we conclude that $S_t=0$ for every $t \in (0,1)$, hence $T_t=(\Phi_t)_* \overline{T}$ for every $t \in (0,1)$.
	\end{proof}

	\section{Uniqueness for the continuity equation} \label{sc:continuity}
	
	In this section we present a simplified proof of uniqueness for the transport of $0$-currents, i.e., signed measures advected via the continuity equation. Our proof differs from the classical one that can be found, e.g., in~\cite[Proposition~8.3.1]{AGS}. In fact, the uniqueness proof below is even slightly different to a mere specialisation of Proposition~\ref{prop:uniqueness_currents} to the case $k = 0$. Still, the decomposability bundle plays a key role, albeit in a different way compared to Proposition~\ref{prop:uniqueness_currents}: Instead of using the pushforward characterisation of Lemma~\ref{lemma:pushforward_currents}, we rely on an approximation result from~\cite{AM}, which allows to approach the directional derivatives of Lipschitz functions in the directions of the decomposability bundle with those of $\Crm^1$ functions.
	
	As shown in~\cite[Section~3.4]{BDR}, in the case of $0$-currents the geometric transport equation~\eqref{eq:GTE} reduces to the continuity equation 
	\begin{equation}\label{eq:continuity}
		\frac{\dd}{\dd t} \mu_t + \dive(b\mu_t) = 0
	\end{equation}
	where $(\mu_t)_{t \in(0,1)}$, is a family of signed measures. We understand this equation in the usual distributional sense, i.e.
	\begin{equation}\label{eq:continuity_distro}
		\int_0^1 \int_{\R^d} \partial_t \psi(t,x) + b(x) \cdot \nabla \psi(t,x) \;\dd \mu_t(x) \;\dd t = 0
	\end{equation}
	for all $\psi \in \Crm^1_c((0,1) \times \R^d)$.
	Setting $\mu := \mathscr L^1(\dd t) \otimes \mu_t(\dd x)$, two equivalent ways of formulating the PDE are the following:
	\begin{equation}\label{eq:PDE_nabla_tilde}
		\int_{(0,1) \times \R^d} (1,b(x)) \cdot \tilde{\nabla} \psi(t,x)\;  \dd \mu(t,x) = 0
	\end{equation}
	for all $\psi \in \Crm_c^1((0,1) \times \R^d)$,
	where $\tilde\nabla \psi(t,x):=(\partial_t \psi(t,x),\nabla\psi(t,x))$. Equivalently,
	\begin{equation}\label{eq:PDE_d}
		\int_{(0,1) \times \R^d} D\psi(t,x)[(1,b(x))]\;\dd \mu(t,x) = 0
	\end{equation}
	for all $\psi \in \Crm_c^1((0,1) \times \R^d)$.

	\begin{proposition}[Uniqueness]\label{prop:uniqueness_measures}
		Let $b\colon \R^d \to \R^d$ be a globally bounded and Lipschitz vector field, and let $\Phi$ be the corresponding flow.  Moreover, let $\overline \mu $ be a signed measure on $\R^d$. Then, a solution to
		\[
		\left\{\begin{aligned}
			\frac{\dd}{\dd t}\mu_t+\dive(b\mu_t) &=0,  \qquad t \in (0,1),\\
			\mu_0 &=\overline\mu
		\end{aligned}\right.
		\]
		is unique in the class of signed measures. Moreover, the solution is given by $\mu_t=(\Phi_t)_\#\overline \mu $.
	\end{proposition}

	\subsection{A simple argument when $b$ is $\Crm^1$}
	
	We illustrate our proof idea by showing uniqueness under the additional regularity assumption that $b$ is $\Crm^1$. The key is to show directly that the solution is necessarily given by
	\[
	\mu_t=(\Phi_t)_\#\overline{\mu} 
	\]
	or, equivalently, that given a solution $(\mu_t)_{t \in (0,1)}$, the map $t\mapsto (\Phi_{-t})_\#\mu_t$ is constant. It is therefore natural to test the weak formulation~\eqref{eq:continuity_distro} with a function of the form $\psi(t,x)=\alpha(t)\beta(\Phi_{-t}(x))$. Since $b$ is $\Crm^1$ then also the flow $\Phi$ is $\Crm^1$, and thus one can differentiate $\psi$ classically. On the one hand we have that
	\begin{align*}
		\partial_t \psi(t,x)&=\alpha'(t)\beta(\Phi_{-t}(x))+\alpha(t)\nabla \beta (\Phi_{-t}(x))\cdot \frac{\dd}{\dd t}\Phi_{-t}(x)\\
		&=\alpha'(t)\beta(\Phi_{-t}(x))-\alpha(t)\nabla \beta (\Phi_{-t}(x))\cdot b(\Phi_{-t}(x)),
	\end{align*}
	where we used the defining property of $\Phi$. On the other hand,
	\begin{align*}
		b(x)\cdot\nabla_x \psi(t,x)&=\alpha(t)b(x)\cdot\big(\nabla \beta(\Phi_{-t}(x))D\Phi_{-t}(x)\big)\\
		&=\alpha(t)\nabla \beta(\Phi_{-t}(x))\cdot\big(D\Phi_{-t}(x)[b(x)]\big)\\
		&=\alpha(t)\nabla \beta(\Phi_{-t}(x))\cdot b(\Phi_{-t}(x)),
	\end{align*}
	where we used Lemma \ref{lemma:directional_derivative_flow} Point $(i)$. 
	Plugging the two terms in the weak formulation gives that
	\[
	\int_0^1 \alpha'(t) \, \dprb{ \mu_t,\beta(\Phi_{-t}(x))} \;\dd t=0
	\]
	for all $\alpha\in\Crm^1((0,1))$ and  all $\beta\in\Crm^1(\R^d)$.
	From this we deduce that $t\mapsto (\Phi_{-t})_\#\mu_t$ is constant, as required.

	\subsection{The Lipschitz case}
	We now consider the general case when $b$ is Lipschitz. Before presenting the proof we need the following key result.
	
	\begin{lemma}\label{lemma:AM}\phantom{.}
		\begin{enumerate}
			\item[(i)] In the setting above, $(1,b(x)) \in V(\mu,(t,x))$ for $\mu$-a.e.\ $(t,x)$. Therefore, every Lipschitz function $\psi:(0,1)\times\R^d\to\R$ is differentiable in direction $(1,b(x))$ for $\mu$-a.e.\ $(t,x)\in(0,1)\times\R^d$.
			\item[(ii)] The equation
			\[
			\int_{(0,1) \times \R^d} D\psi(t,x)[(1,b(x))]\;\dd \mu(t,x) = 0
			\]
			holds for every $\psi \in \Lip((0,1) \times \R^d)$, where the integral is well-defined by~(i).
		\end{enumerate}
	\end{lemma}
	
	We remark that, as a consequence of Point $(ii)$ above, one can use \emph{any} Lipschitz test function in the distributional formulation.
	
	\begin{proof}
		\noindent\textit{Ad~(i).} Let $U$ be the $1$-current in $(0,1)\times\R^d$ defined by $U:=(1,b(x)) \mu$. Then~\eqref{eq:PDE_nabla_tilde} can be understood as
		\[
		\partial U=0\qquad\text{in $(0,1)\times\R^d$,}
		\]
		hence $U$ is a normal 1-current without boundary in $(0,1)\times\R^d$. Thus, the assertion follows by~\eqref{eq:span_in_bundle}.
		
		\medskip\noindent\textit{Ad~(ii).} By~\cite[Corollary~8.3]{AM} there exists a sequence of functions $(\psi_j)_j \subset \Crm^1((0,1) \times \R^d)$ such that $\psi_j\to \psi$ uniformly, $\sup_j \Lip(\psi_j)<\infty$, and 
		\[
		D\psi_j(t,x)[(1,b(x))] \to D\psi(t,x)[(1,b(x))]\qquad\text{ for $\mu$-a.e.\ $(t,x)\in(0,1)\times\R^d$.}
		\]
		Resorting to~\eqref{eq:PDE_d} and the dominated convergence theorem we reach the conclusion.
	\end{proof}

	We finally come to the uniqueness proof.
	
	\begin{proof}[Proof of Proposition~\ref{prop:uniqueness_measures}]
		We choose as test function $\psi(t,x) := \alpha(t) \beta(\Phi_{-t}(x))$, where $\alpha \in \Crm^1_c((0,1))$ and $\beta \in \Crm^1_c(\R^d)$. Recalling that $\Phi_{-t}$ is Lipschitz, also $\psi$ is a Lipschitz function and therefore, by Lemma~\ref{lemma:AM}~(ii),
		\begin{equation}\label{eq:PDE_again}
			\int_{(0,1) \times \R^d} D\psi(t,x)[(1,b(x))]\;\dd \mu(t,x) = 0.
		\end{equation}
		Let us now compute directly the directional derivative $D\psi (t,x)[(1,b(x))]$. We claim that 
		\begin{equation}\label{eq:directional_derivative_claim}
			D\psi (t,x)[(1,b(x))]=\alpha'(t)\beta(\Phi_{-t}(x)) \qquad\text{for every $(t,x)\in(0,1)\times\R^d$.}
		\end{equation}
		We observe in passing that, for this particular choice of $\psi$, the computation below directly shows that this directional derivative exists at \textit{every} point, and not just $\mu$-almost everywhere (which instead is guaranteed for \textit{any} Lipschitz function by Lemma~\ref{lemma:AM}~(i)). We have 
		\begin{align*}
			D(\alpha(t)\beta(\Phi_{-t}(x)))[(1,b(x))]
			&=  \lim_{h \to 0} \frac{\alpha(t+h)\beta(\Phi_{-t-h}(x+hb(x))) -\alpha(t) \beta(\Phi_{-t}(x)))}{h} \\
			&=  \lim_{h \to 0} \frac{\alpha(t+h) - \alpha(t)}{h}  \beta(\Phi_{-t}(x)) \\
			&\qquad + 
			\lim_{h \to 0} \frac{\alpha(t+h)[\beta(\Phi_{-t-h}(x+hb(x))) - \beta(\Phi_{-t}(x)))}{h} \\
			&=  \alpha'(t)\beta(\Phi_{-t}(x)) + \ell. 
		\end{align*}
		Let us prove that $\ell=0$. 
		Since $\beta\in\Crm^1_c(\R^d)$, we deduce that
		\begin{align*}
			| \beta(\Phi_{-t-h}(x+hb(x))) - \beta(\Phi_{-t}(x))) | & \le \Lip(\beta) |\Phi_{-t-h}(x+hb(x))) - \Phi_{-t}(x)| \\ 
			& \le \Lip(\beta) \Lip(\Phi_{-t-h}) \|b\|_\infty \Lip(b)  h^2,
		\end{align*} 
		where we have used the semigroup law $\Phi_{-t}(x)=\Phi_{-t-h}(\Phi_h(x))$ and~\eqref{eq:conto_fondamentale}.
		Using this inequality and the fact that $\Lip(\Phi_{-t-h})$ is uniformly bounded in $h$ we conclude that $\ell=0$, and therefore~\eqref{eq:directional_derivative_claim} is proven. Thus, plugging~\eqref{eq:directional_derivative_claim} into~\eqref{eq:PDE_again}, we have
		\begin{align*}
			0  &= \int_{(0,1) \times \R^d} \alpha'(t)\beta(\Phi_{-t}(x)) \;\dd \mu(t,x) \\
			&= \int_0^1 \alpha'(t) \, \dprb{ \mu_t, \beta\circ \Phi_{-t} }  \;\dd t  \\
			&= \int_0^1 \alpha'(t) \, \dprb{ (\Phi_{-t})_{\#}\mu_t, \beta }  \;\dd t.
		\end{align*}
		This holds for every $\alpha\in \Crm^1_c((0,1))$, hence we obtain that the map $t \mapsto \dprb{ (\Phi_t)^{-1}_{\#}\mu_t, \beta }$ is constant for every $\beta$. Recalling the weak$^*$ continuity of $t\mapsto \mu_t$ it follows that
		\[
		\dprb{ (\Phi_t)^{-1}_{\#}\mu_t, \beta } =\dprb{ \mu_0, \beta }\qquad\text{for all $\beta\in \Crm^1_c(\R^d)$},
		\]
		and since $\beta$ is arbitrary we have $(\Phi_t)^{-1}_{\#}\mu_t = \mu_0$ as measures, i.e. $\mu_t = (\Phi_t)_{\#} \mu_0$.
	\end{proof}

	\begin{remark}\label{rem:no_dec_bundle} In the proof of Proposition~\ref{prop:uniqueness_measures}, the use of the decomposability bundle can be avoided.
		Indeed, for the specific test function employed in the proof, the computations therein directly show that the directional derivative $D\psi(t,x)[(1,b(x))]$ exists at every $(t,x)$ (and not just $\mu$-almost everywhere). In addition, it is easy to construct smooth approximations $\psi_j$ of $\psi$ satisfying~\cite[Corollary~8.3]{AM}: it is enough to consider smooth approximations $b^j$ of the vector field $b$ (for instance, by convolution with a kernel) and define 
		\[
		\psi_j(t,x):=\alpha(t)\beta(\Phi^j_{-t}(x)), \qquad (t,x) \in (0,1) \times \R^d
		\]
		where $\Phi^j$ denotes the flow of $b^j$. This sequence can be used to check that the distributional formulation holds true with the Lipschitz test function $\psi$. 
	\end{remark}

	\bibliographystyle{plain}
	\bibliography{biblio}
	
\end{document}